\documentclass{article}
\usepackage{amsmath}
\usepackage{amssymb}

\newcommand{\C}{\mathbb{C}}

\newcommand{\N}{\mathbb{N}}
\newcommand{\TT}{\mathbb{T}}

\def\co{\colon}
\def\la{\langle}
\def\ra{\rangle}
\def\i{\infty}

\def\op{\oplus}

\def\HH{{\cal H}}
\def\L{{\cal L}}

\def\G{{\cal G}}
\def\NN{{\cal N}}
\def\E{{\cal E}}
\def\M{{\cal M}}
\def\F{{\cal F}}

\def\X{{\cal X}}

\def\s{\sigma}

\def\D{\Delta}
\def\l{\lambda}
\def\m{\mu}

\def\sp{\sigma_{\rm p}}

\def\r{\rho}

\def\x{\xi}

\def\1{\mathbf 1}

\def\wt{\widetilde}

\def\wh{\widehat}

\def\ot{\otimes}

\def\G{\Gamma}

\newtheorem{pro}{Proposition}
\newtheorem{thm}[pro]{Theorem}
\newtheorem{lem}[pro]{Lemma}

\newtheorem{que}[pro]{Question}

\def\rank{\hbox{\rm rank\,}}
\def\Lat{\hbox{\rm Lat\,}}

\begin{document}

\title{Adjacent cross-sections of the commutant of Hilbert space operators}

\author{L\'aszl\'o K\'erchy}

\date{}

\maketitle

\begin{abstract}
\noindent
Applying the techniques resulting the existence of almost invariant half-spaces, similarity models $\wh T$ can be given for upper triangular operator-matrices 
$T= \left[\begin{matrix}A&C\\
0&B\end{matrix}\right]$. 
The model $\wh T$ is also an operator-matrix, containing two diagonal operators in the general case \cite{ker25}, and the unilateral shift $S$ together with a diagonal operator in the particular case when $A$ is similar to $S$ \cite{ker26}.
Well-chosen compressions of operators in the commutant $\{\wh T\}'$ form a linear manifold $\wh\L$ satisfying the condition that every $\wh X\in\wh \L$ is transformed into $\wh Y$ with $\rank\wh Y\le 2$ by a canonical mapping.
Furthermore, a cyclcity property of $\{T\}'$ yields transitivity of $\wh\L$.
In \cite{ker25} and \cite{ker26} the 3-dimensional cross-sections of $\wh\L$ have been investigated characterizing the canonical bases occurring in the corresponding subspaces of the matrix-algebra $M_3[\C]$.
In this paper new conditions are provided by studying matching of adjacent cross-sections.

\noindent
\emph{AMS Subject Classification} (2020): 47A08, 47A15, 47A20, 47L05.

\noindent
\emph{Key words}: invariant subspace, hyperinvariant subspace, compression of operator, unilateral shift, transitive subspace of matrices, singular matrices.
\end{abstract}

\section{Introduction}
\label{introduction}

Let $\HH_1$ and $\HH_2$ be complex Hilbert spaces with $\dim\HH_1= \dim\HH_2=\aleph_0$ and let us consider a (bounded, linear) operator
\[T= \left[\begin{matrix} A& C\\
0&B\end{matrix}\right]\in\L(\HH_1\op\HH_2).\]
It is clear that the subspace $\HH_1(\equiv\HH_1\op\{0\})$ is invariant for $T : \, T\HH_1\subset\HH_1$.
The commutant $\{T\}'$ of $T$ consists of those operators $Q\in\L(\HH_1\op\HH_2)$ which commute with $T :\, QT=TQ$.
It will be assumed that $T$ is not scalar multiple of the identity.
We are interested in the following basic question:

\begin{que}
\label{basic}
Does $T$ have a nontrivial hyperinvariant subspace, that is a subspace $\M\in \hbox{\rm Lat}(\HH_1\op\HH_2)$, different from $\{0\}\op\{0\}$ and $\HH_1\op\HH_2$, which is invariant for every $Q\in\{T\}' :\, Q\M\subset\M$?
\end{que}

The hyperinvariant subspaces of $T$ form the complete lattice $\hbox{Hlat\,}T$.
So the question above asks whether $\hbox{Hlat\;}T\ne \left\{\{0\}\op\{0\}, \HH_1\op\HH_2\right\}$.
Notice that an affirmative answer for Question\,\ref{basic} yields that every nonscalar operator $A$ has a proper hyperinvariant subspace.
Indeed, take $T= A\op A$ and observe that $\hbox{Hlat\,}T= \left\{\NN\op\NN: \NN\in \hbox{Hlat\,}A\right\}$.
In that way we would obtain positive answers for the hyperinvariant subspace problem (HSP), and in particular for the invariant subspace problem (ISP), which are arguably the most challenging open questions in operator theory.

A recent approach to (ISP) in an arbitrary (infinite dimensional) Banach space $\X$ was initiated in \cite{APTT}, and culminated in \cite{Tc} proving that every operator $T$ on $\X$ has an \emph{almost invariant half-space}. 
We recall that $\M$ is a half-space in $\X$ if $\M$ and $\X/\M$ are infinite dimensional.
Furthermore, $\M$ is almost invariant for $T$ if $\M$ is invariant for a finite rank perturbation of $T$.
The technique of this approach was refined in the Hilbert space setting in \cite{JKP}, which led to a structure theorem in the paper \cite{HP}.

Applying this method to our block-triangular operator
$T\in\L(\HH_1\op\HH_2)$ more delicate similarity model $\wh T\in\L(\HH^{(4)})$ has been constructed for a translate of $T$ in \cite{ker25}.
Here $\HH$ is an auxiliary complex Hilbert space with $\dim\HH=\aleph_0$ and $\HH^{(4)}= \HH\op\HH\op\HH\op\HH$.
The model $\wh T$ contains two diagonal components: $D, D_*$.
The central linear manifold $\wh\L\subset\L(\HH)$ consists of the entries of the operators in $\{\wh T\}'$ belonging to $\L(\hbox{dom\,}D, \hbox{dom\,}D_*)$.
The crucial property of $\wh\L$ is that $\rank(\wh XD-D_*\wh X)\le 2$ holds, for every $\wh X\in\wh\L$.
Furthermore, a cyclicity condition of $\{T\}'$, which can be assumed in the quest for proper hyperinvariant subspaces, implies that $\wh\L$ is transitive.
An orthonormal basis $\{e_k\}_{k=1}^\i$ is fixed in $\HH$.
For every $r\in\N$, we consider the subspace $\E_{3,r}= \vee\{e_r,e_{r+1}, e_{r+2}\}$.
The cross-section $\wh\L_r$ consists of the compressions of the operators in $\wh\L$ to $\E_{3,r}$.
This subspace of $\L(\E_{3,r})$ inherits the rank-2 property and the transitivity of $\wh\L$.
The algebra-isomorphism $M_{3,r}\co \L(\E_{3,r}) \to M_3[\C]$ assigns to each operator its matrix in the orthonormal basis $(e_r, e_{r+1}, e_{r+2})$.
The main result in \cite{ker25} characterizes the arising subspaces $M_{3,r}(\wh\L_r)$ in terms of properties of a canonical basis.

In \cite{ker26} the case when $A$  is similar to the unilateral shift $S$ was investigated.
Now the similarity model $\wh T\in\L(\HH^{(3)})$ of a translate of $T$ contains $S$ and a diagonal operator $D_*$.
The linear manifold $\wh\L\subset\L(\HH)$ consists of those entries of the operators in $\{\wh T\}'$ which belong to $\L(\hbox{dom\,}S, \hbox{dom\,}D_*)$.
For every $\wh X\in\wh\L$, we have $\rank(\wh XS-D\wh X)\le1$, where $D$ is a translate of $D_*$.
Furthermore, a cyclicity condition of $\{T\}'$ implies that $\wh\L$ is transitive.
For every $r\in\N$, let $\wh\L_r$ be again the compression of $\wh\L$ to the subspace $\E_{3,r}= \vee\{e_r,e_{r+1}, e_{r+2}\}$.
The cross-section $\wh\L_r$ inherits the transitivity of $\wh\L$, and $\rank\G_r(X)\le2$ is true for every $X\in\wh\L_r$ with a canonical mapping $\G_r\in \L(\L(\E_{3,r}))$.
The main result of \cite{ker26} provides complete characterization of the arising subspaces $M_{3,r}(\wh\L_r)$ in terms of conditions concerning the elements of a canonical basis.

Our aim in this paper is to give further conditions for these bases by studying matching of adjacent cross-sections.

In Section\,2 compressions of operators are studied.
In Section\,3 we examine the general case, while in Sections\,4 and 5 we discuss the case when $A$ is similar to $S$.
The conditions obtained simplify the featured canonical bases in the compressions, and establish connections between the parameters in the neighboring cross-sections.
If any of the resulting conditions fails, then $T$ has a nontrivial hyperinvariant subspace; more precisely, there exists a nonzero vector $x\in\HH_1$ such that $\{T\}'x$ is not dense in $\HH_1\op\HH_2$.
This is because the rules concerning ranks cannot be broken.

\section{Compressions}
\label{compression}

Let $\{e_k\}_{k=1}^\i$ be an orthonormal basis in the Hilbert space $\HH$.
For any subspace $\E\in \Lat\HH$, the \emph{compression mapping}
\[C_{\HH,\E}\co \L(\HH)\to\L(\E),\; X\mapsto P_\E X\big|_\E\]
is a bounded linear transformation.
(Here $P_\E\in\L(\HH)$ stands for the orthogonal projection onto $\E$.)
Given any $\F\in\Lat\E$, we may consider the mapping
\[C_{\E,\F}\co \L(\E) \to \L(\F), \; X\mapsto P_\F X\big|_\F.\]
The following simple statement establishes connection among these transformations.

\begin{lem}
\label{composition}
We have $\; C_{\E,\F} \circ C_{\HH,\E} = C_{\HH,\F}$.
\end{lem}

\noindent
{\bf Proof.}
For completeness we sketch the easy proof.
Given any $X\in\L(\HH)$, let $Y= C_{\HH,\E}(X)\in\L(\E)$ and $Z= C_{\E,\F}(Y)\in\L(\F)$.
For every $x\in\F$ we have $Zx= P_\F Yx= P_\F P_\E Xx= P_\F Xx$, and so $Z= C_{\HH,\F}(X)$.

\rightline{$\square$}

We are interested in diagonal subspaces spanned by the basis vectors.
For any $r\in\N$, let
\[\E_{3,r}:= \vee\{e_r,e_{r+1}, e_{r+2}\} \quad \hbox{ and } \quad \F_{2,r}:= \vee\{e_r, e_{r+1}\}.\]
The \emph{matricial mappings}
\[M_{3,r}\co \L(\E_{3,r})\to M_3[\C] \quad \hbox{ and } \quad M_{2,r}\co \L(\F_{2,r})\to M_2[\C]\]
assign to each operator its matrix with respect to the given orthonormal basis.
These are algebra-isomorphisms.
For any $k,l\in\N_3(=\{1,2,3\})$, we consider the \emph{partial mapping}
\[P_{k,l}\co M_3[\C] \to M_2[\C], \quad X=[\x_{i,j}]_3\mapsto X_{k,l},\]
where $X_{k,l}$ is obtained from $X$ by deleting its $k$-th row and $l$-th column.
The following statement will play central role in the sequel.

\begin{lem}
\label{partial}
For every $r\in\N$, we have
\[P_{1,1}\circ M_{3,r}\circ C_{\HH,\E_{3,r}}= M_{2,r+1}\circ C_{\HH,\F_{2,r+1}}= P_{3,3}\circ M_{3,r+1}\circ C_{\HH,\E_{3,r+1}}.\]
\end{lem}

\noindent
{\bf Proof.}
Given any $\wt X\in\L(\HH)$, let us consider its matrix $[\wt X]= [\x_{i,j}]_\i$, where
\[\x_{i,j}= \la e_i,\wt X e_j\ra \quad (i,j\in\N).\]
For any $r\in\N$, let $\wt X_r:= C_{\HH,\E_{3,r}}(\wt X)$ and $X_r= [\x_{i,j}^r]_3= M_{3,r}(\wt X_r)$.
Since
\[\x_{i,j}^r = \la e_{r+i-1}, \wt X_r e_{r+j-1}\ra = \la e_{r+i-1}, \wt X e_{r+j-1}\ra = \x_{r+i-1,r+j-1} \quad (i,j\in \N_3),\]
it follows that
\[P_{1,1}(X_r)= \left[\begin{matrix}
\x_{2,2}^r & \x_{2,3}^r\\
\x_{3,2}^r & \x_{3,3}^r\end{matrix}\right] = \left[\begin{matrix}
\x_{r+1,r+1}& \x_{r+1,r+2}\\
\x_{r+2,r+1}&\x_{r+2,r+2}\end{matrix}\right] =: X_0.\]
On the other hand, let $\wt Y_{r+1} := C_{\HH,\F_{2,r+1}}(\wt X)$ and $Y_{r+1} = [\eta_{i,j}^{r+1}]_2= M_{2,r+1}(\wt Y_{r+1})$.
Since $\eta_{i,j}^{r+1}= \x_{r+i,r+j}\; (i,j\in\N_2)$, it follows that $Y_{r+1}=X_0$.

Finally, take $\wt X_{r+1}= C_{\HH,\E_{3,r+1}}(\wt X)$ and $X_{r+1}= [\x_{i,j}^{r+1}]_3 = M_{3, r+1}(\wt X_{r+1})$, where $\x_{i,j}^{r+1}= \x_{r+i,r+j}\; (i,j\in\N_3)$.
Then
\[P_{3,3}(X_{r+1})= \left[\begin{matrix}
\x_{1,1}^{r+1} & \x_{1,2}^{r+1}\\
\x_{2,1}^{r+1}& \x_{2,2}^{r+1}\end{matrix}\right] = \left[\begin{matrix}
\x_{r+1,r+1}& \x_{r+1,r+2}\\
\x_{r+2,r+1}& \x_{r+2,r+2}\end{matrix}\right]= X_0,\]
and the proof is completed.

\rightline{$\square$}

\section{General operators}
\label{general}

Let us consider an arbitrary nonscalar operator of the form
\[T= \left[\begin{matrix}
A & C\\
0&B\end{matrix}\right] \in \L(\HH_1\op\HH_2).\]
In view of (HSP) we may assume that $T$ and its adjoint $T^*$ don't have eigenvalues: $\sp(T)=\sp(T^*)=\emptyset$.
Then, clearly $\sp(A)=\sp(B^*)=\emptyset$ hold too.
Let us select points on the boundaries of the spectra: $\l_0\in\partial\s(A)$ and $\m_0\in\partial\s(B)$.
These points will be approximated from the resolvent sets $\r(A)= \C\setminus \s(A)$ and $\r(B)= \C\setminus\s(B)$, respectively.
Let $\{e_k\}_{k=1}^\i$ be a fixed orthonormal basis in an auxiliary Hilbert space $\HH$.
It can be shown that $T-\m_0I$ is similar to an operator $\wh T\in\L(\HH^{(4)})$ of the form
\[\wh T= \left[\begin{matrix}
D&*&*&*\\
F&*&*&*\\
0&0&*&*\\
0&0&F_*&D_*\end{matrix}\right],\]
where
\[D=\sum_{k=1}^\i (\l_k-\m_0) e_k\ot e_k \; \hbox{ and }\; \r(A)\ni \l_k\to \l_0\; \hbox{ sufficiently  fast},\]
\[D_*= \sum_{k=1}^\i (\m_k-\m_0)e_k\ot e_k \; \hbox{ and } \; \r(B)\ni \m_k\to\m_0 \; \hbox{ sufficiently fast};\]
moreover $\{\l_k\}_{k=1}^\i\cup\{\m_k\}_{k=1}^\i$ consists of distinct nonzero points, and $\rank F = \rank F_*=1.$
It is clear that $\hbox{Hlat\,}T= \hbox{Hlat}(T-\m_0I) \simeq \hbox{Hlat\,}\wh T$.

Let us consider the linear manifold
\[\wh\L= \left\{ Q_{4,1}: [Q_{i,j}]_4\in\{\wh T\}'\right\}\]
in $\L(\HH)$.
The following rank-2 condition of $\wh\L$ plays significant role in our study:
\[\hbox{\rm rank}(D_*\wh X-\wh XD)\le 2 \quad \hbox{ holds for every }\; \wh X\in\wh\L.\]
In view of (HSP) we can make also the following assumption:

\medskip
\noindent
\textup{(H)} $\quad\quad \{T\}'x$ is dense in $\HH_1\op\HH_2$, for every nonzero $x\in\HH_1$.

\medskip
\noindent
(Notice that this situation occurs when the bilateral shift acts on the function space $L^2(\TT)= H^2 \op H^2_-$.)
The assumption (H) readily implies that $\wh\L$ is transitive.

For brevity we introduce the notation:
\[b_i= \m_i-\m_0, \; a_j=\l_j-\m_0, \; \wh c_{i,j}= b_i-a_j\; (i,j\in\N); \; \hbox{ and }\; \wh C= [\wh c_{i,j}]_\i.\]
For any $r\in\N$, let $\E_{3,r}$ be as before, and let us consider the cross-section
\[\wh\L_r= C_{\HH,\E_{3,r}}(\wh\L) \in \Lat \L(\E_{3,r}).\]
It is immediate that $\wh\L_r$ is also transitive.
Furthermore the Schur-product $C^r\circ X$ is a singular matrix, for every $X\in M_{3,r}(\wh\L_r)$.
Here $C^r= [c^r_{i,j}]_3$ with $c^r_{i,j}= \wh c_{r+i-1,r+j-1}\; (i,j\in\N_3)$.

In order to provide characterization of these cross-sections we need additional concepts.
Let $C=[c_{i,j}]_3\in M_3[\C_0]$ be arbitrary, where $\C_0= \C\setminus\{0\}$.
We recall that $C_{k,l}= P_{k,l}(C)\; (k,l\in\N_3)$.
For any $Y=[\eta_{i,j}]_2\in M_2[\C_0]$, let
\[\r(Y):= \frac{\eta_{1,1}\eta_{2,2}}{\eta_{1,2}\eta_{2,1}}.\]
We shall say that a subspace $\L\in \Lat M_3[\C]$ is \emph{$C$-normal}, if $\dim\L=5$ and there exists a basis $(L_1,L_2,Q_1,Q_2,Q_3)$ in $\L$ of the following form:
\[L_1=\left[\begin{matrix}
1&0&0\\
\r(C_{3,2})q_2 &0&0\\
\r(C_{2,2})q_3&0&0\end{matrix}\right], \quad L_2= \left[\begin{matrix}
0&1&0\\
0&\r(C_{3,1})q_2&0\\
0&\r(C_{2,1})q_3&0\end{matrix}\right],\]
\[Q_1=\left[\begin{matrix}
0&0&1\\
0&0&q_2\\
0&0&q_3\end{matrix}\right], \;
Q_2= \left[\begin{matrix}
0&0&0\\
p_1&p_2&1\\
0&0&0\end{matrix}\right], \; Q_3=
\left[\begin{matrix}0&0&0\\
0&0&0\\
\r(C_{1,2})p_1& \r(C_{1,1})p_2&1\end{matrix}\right],\]
with parameters $p_1, p_2, q_2, q_3\in\C_0$.
The general element in $\L$ is of the form
\[X=\left[\begin{matrix}
z_1&z_2&w_1\\
z_1\r(C_{3,2})q_2+w_2p_1 & z_2\r(C_{3,1})q_2+w_2p_2&w_2\\
z_1\r(C_{2,2})q_3+w_3\r(C_{1,2})p_1& z_2\r(C_{2,1})q_3 +w_3\r(C_{1,1})p_2&w_3\end{matrix}\right],\]
where $(z_1,z_2,w_1,w_2,w_3)$ runs through $\C^5$.

We know from \cite{ker25} that \emph{the linear manifold $\wh\L$ is $\wh C$-normal}, which means that the subspace $M_{3,r}(\wh\L_r)$ is $C^r$-normal
(with parameters $p_{1,r}, p_{2,r}, q_{2,r}, q_{3,r}$), for every $r\in\N$.
Applying Lemma\,\ref{partial}, we may obtain connection between the parameters of adjacent cross-sections.

\begin{thm}
\label{connection}
For every $r\in\N$, we have
\begin{itemize}
\item[\textup{(i)}] $\; q_{3,r}= q_{2,r} q_{2,r+1} \r(C^{r+1}_{3,1}),$
\item[\textup{(ii)}] $\; p_{1,r+1}= p_{2,r}p_{2,r+1} \r(C^r_{1,1}).$
\end{itemize}
\end{thm}

\noindent
{\bf Proof.}
In virtue of the general forms of the matrices in the subspaces $M_{3,r}(\wh\L_r)$ and $M_{3,r+1}(\wh\L_{r+1})$, we obtain by Lemma\,\ref{partial} that for every ($z_1,z_2,w_1,w_2,w_3)\in\C^5$ there exists corresponding $(\wt z_1, \wt z_2, \wt w_1, \wt w_2, \wt w_3)\in\C^5$ such that
\[\left[\begin{matrix}
z_2 \r(C_{3,1}^r)q_{2,r}+w_2 p_{2,r} & w_2\\
z_2 \r(C_{2,1}^r)q_{3,r} + w_3 \r(C_{1,1}^r)p_{2,r}& w_3\end{matrix}\right] = \quad\quad\quad\quad\quad\quad\quad\quad\]
\[\quad\quad\quad\quad\quad = \left[\begin{matrix}
\wt z_1 & \wt z_2\\
\wt z_1\r(C_{3,2}^{r+1}) q_{2,r+1}+ \wt w_2 p_{1,r+1}& \wt z_2 \r(C_{3,1}^{r+1}) q_{2,r+1} + \wt w_2 p_{2,r+1}\end{matrix}\right].\]
The entrywise equations are
\begin{itemize}
\item[\textup{(1)}] $\;\wt z_1 = z_2 \r(C_{3,1}^r)q_{2,r} + w_2 p_{2,r},$
\item[\textup{(2)}] $\; \wt z_2=w_2,$
\item[\textup{(3)}] $\;\wt z_1 \r(C_{3,2}^{r+1})q_{2,r+1} + \wt w_2 p_{1,r+1} = z_2 \r(C_{2,1}^r) q_{3,r} + w_3 \r(C_{1,1}^r) p_{2,r},$
\item[\textup{(4)}] $\; \wt z_2 \r(C_{3,1}^{r+1}) q_{2,r+1} + \wt w_2 p_{2,r+1} = w_3$.
\end{itemize}
Eliminating $\wt w_2$, we may obtain from (3) and (4):
\begin{itemize}
\item[\textup{(5)}] $\; \wt z_1\r(C_{3,2}^{r+1}) q_{2,r+1} p_{2,r+1} - \wt z_2 \r(C_{3,1}^{r+1}) q_{2,r+1} p_{1,r+1} $
\item[] $\quad\quad\quad
= z_2 \r(C_{2,1}^r) q_{3,r} p_{2,r+1} + w_3 \left(\r(C_{1,1}^r)p_{2,r} p_{2,r+1} - p_{1,r+1}\right)$.
\end{itemize}
Substituting (1) and (2) into (5) yields
\begin{eqnarray*}
& &z_2\left(\r(C_{3,1}^r) q_{2,r} \r(C_{3,2}^{r+1}) q_{2,r+1}p_{2,r+1} - \r(C_{2,1}^r) q_{3,r}p_{2,r+1}\right)\\
  & & \quad\quad+ w_2\left(p_{2,r}\r(C_{3,2}^{r+1}) q_{2,r+1}p_{2,r+1} - \r(C_{3,1}^{r+1}) q_{2,r+1}p_{1,r+1}\right) \\
 & & \quad\quad - w_3\left(\r(C_{1,1}^r) p_{2,r}p_{2,r+1} - p_{1,r+1}\right)\\
 & & \quad =0.
\end{eqnarray*}
Since $(z_2,w_2,w_3)\in \C^3$ can be freely chosen, the coefficients here must be zero, and so
\begin{itemize}
\item[\textup{(6)}] $\; q_{3,r}= q_{2,r} q_{2,r+1} \r(C_{3,1}^r) \r(C_{3,2}^{r+1}) \r(C_{2,1}^r)^{-1}$,
\item[\textup{(7)}] $\; p_{1,r+1}= p_{2,r}p_{2,r+1} \r(C_{3,2}^{r+1}) \r(C_{3,1}^{r+1})^{-1},$
\item[\textup{(8)}] $\; p_{1,r+1}= p_{2,r} p_{2,r+1} \r(C_{1,1}^r)$.
\end{itemize}
Direct computations show that
\[\r(C_{3,1}^r) \r(C_{3,2}^{r+1}) \r(C_{2,1}^r)^{-1}= \r(C_{3,1}^{r+1}), \quad \r(C_{3,2}^{r+1}) \r(C_{3,1}^{r+1})^{-1}= \r(C_{1,1}^r).\]
Thus, (6)--(8) imply that (i) and (ii) hold.

\rightline{$\square$}

\section{Operators containing a unilateral shift}
\label{shift}

Let us consider again a block-triangular operator
\[T=\left[\begin{matrix}
A&C\\
0&B\end{matrix}\right]\in\L(\HH_1\op\HH_2).\]
Let us fix an orthonormal basis $\{e_k\}_{k=1}^\i$ in the auxiliary Hilbert space $\HH$.
Let us assume that \emph{$A$ is similar to the unilateral shift}
\[S=\sum_{k=1}^\i e_{k+1}\ot e_k\in\L(\HH).\]
We assume also that $\sp(T^*)=\emptyset$, whence $\sp(B^*)=\emptyset$ follows.
Select any point $b_0\in\partial\s(B)$.
Then $T-b_0I$  will be similar to an operator $\wh T\in\L(\HH^{(3)})$ of the form
\[\wh T= \left[\begin{matrix}
S-b_0I&*&*\\
0&*&*\\
0&F_*&D_*\end{matrix}\right],\]
where $D_*= \sum_{k=1}^\i (b_k-b_0) e_k\ot e_k, \; \r(B)\ni b_k\to b_0$ sufficiently fast, $\{b_k\}_{k=1}^\i$ are distinct nonzero complex numbers, and $\rank F_*\le1$.
Let us consider the linear manifold
\[\wh\L= \left\{ Q_{3,1}: [Q_{i,j}]_3\in\{\wh T\}'\right\}\]
in $\L(\HH)$.
Now
\[\hbox{ rank}(D\wh X- \wh XS)\le 1\] holds for every $\wh X \in\wh\L$, where
\[D=D_*+b_0I= \sum_{k=1}^\i b_k e_k\ot e_k.\]
We assume also that the former cyclicity condition (H) is valid, which implies that $\wh\L$ is transitive.

For any $r\in\N$, let $\E_{3,r}$ be as before, and let us consider the transitive cross-section
\[\wh\L_r= C_{\HH,\E_{3,r}}(\wh\L) \in\Lat\L(\E_{3,r}).\]
Forming the diagonal matrix $\D_r= \hbox{diag}(b_r, b_{r+1}, b_{r+2})\in M_3[\C]$, we may deduce that
\[\hbox{\rm rank}(\D_rX-XJ_3)\le 2\]
is true, for every $X\in M_{3,r}(\wh\L_r)$.
Here $J_3$ stands for the Jordan-cell
\[J_3= \left[\begin{matrix}
0&0&0\\
1&0&0\\
0&1&0\end{matrix}\right].\]
(Notice that the subspace $\E_{3,r}$ is not invariant for $S$.)

Given distinct numbers $b_1,b_2,b_3\in\C_0$, let $\D= \hbox{diag}(b_1,b_2,b_3)\in M_3[\C]$.
We introduce also the notation
\[b_\D(1)= \frac{1}{b_3}-\frac{1}{b_2}, \quad b_\D(2)= \frac{1}{b_3}-\frac{1}{b_1}, \quad b_\D(3)= \frac{1}{b_2}-\frac{1}{b_1}.\]
We say that the subspace $\L\in\Lat M_3[\C]$ is \emph{$\D$-normal}, if $\dim\L=5$ and there exists a basis $(L_1,L_2,Q_1,Q_2,Q_3)$ in $\L$ of the form (T1) or (T2), where

\medskip
\noindent
(T1)
\[L_1=\left[\begin{matrix}
1&0&0\\
x&0&0\\
y&0&0\end{matrix}\right], \quad L_2=\left[\begin{matrix}
1&1&0\\
b_\D(3)x&x&0\\
b_\D(2)y&y&0\end{matrix}\right],\]
\[Q_1=\left[\begin{matrix}
0&0&1\\
0&0&0\\
0&0&0\end{matrix}\right], \; Q_2= \left[\begin{matrix}
0&0&0\\
b_\D(3)/b_2& b_\D(3)&1\\
0&0&0\end{matrix}\right], \; Q_3= \left[\begin{matrix}
0&0&0\\
0&0&0\\
b_\D(2)/b_3& b_\D(2)&1\end{matrix}\right],\]
\[\hbox{with parameters }\; x,y\in\C_0;\]
and

\medskip
\noindent
(T2) $\quad L_1, L_2$ as before,
\[Q_1=\left[\begin{matrix}
0&0&1\\
q'&q&0\\
q'y/x+b_\D(1)qy/x& qy/x&0\end{matrix}\right],\]
\[Q_2=\left[\begin{matrix}
0&0&0\\
b_\D(3)/b_2-q'/x& b_\D(3)-q/x&1\\
0&0&0\end{matrix}\right],\]
\[Q_3=\left[\begin{matrix}
0&0&0\\
0&0&0\\
b_\D(2)/b_3-b_\D(1)q/x-q'/x& b_\D(2)-q/x & 1\end{matrix}\right],\]
\[\hbox{with parameters }\; x,y,q\in\C_0\; \hbox{ and }\; q'\in\C.\]

By the results of \cite{ker26} we know that $\wh\L$ is \emph{$D$-normal}, which means that $M_{3,r}(\wh\L_r)$ is $\D_r$-normal (with parameters $x_r,y_r, q_r, q_r'$), for every $r\in\N$.
Examining matching of adjacent cross-sections we are going to show that the type (T1) is actually excluded.

\begin{thm}
\label{exclusion}
If $\wh\L$ is $D$-normal, then $M_{3,r}(\wh\L_r)$ is of type {\rm (T2)}, for every $r\in\N$.
\end{thm}

\noindent
{\bf Proof.}
Suppose that $M_{3,r}(\wh\L_r)$ is of type (T1), for some $r\in\N$.

\medskip
\noindent
\emph{Case 1.}
Let us assume first that $M_{3,r+1}(\wh\L_{r+1})$ is also of type (T1).
The general element in $M_{3,r}(\wh\L_r)$ is of the form
\[\left[\begin{matrix}
z_1&z_2&w_1\\
z_1x_r+z_2b_{\D_r}(3)x_r+w_2 b_{\D_r}(3)/b_{r+1}& z_2x_r+w_2 b_{\D_r}(3)&w_2\\
z_1y_r+z_2b_{\D_r}(2) y_r+ w_3 b_{\D_r}(2)/b_{r+2}& z_2y_r+w_3b_{\D_r}(2)& w_3\end{matrix}\right],\]
with $x_r, y_r\in\C_0$, and $(z_1,z_2,w_1,w_2,w_3)$ runing through $\C^5$.
Similarly, the general element in $M_{3, r+1}(\wh\L_{r+1})$ is of the form
\[\left[\begin{matrix}
\wt z_1& \wt z_2& \wt w_1\\
\wt z_1 x_{r+1} + \wt z_2 b_{\D_{r+1}}(3) x_{r+1} + \wt w_2 b_{\D_{r+1}}(3)/ b_{r+2} & \wt z_2 x_{r+1} + \wt w_2 b_{\D_{r+1}}(3) & \wt w_2\\
\wt z_1 y_{r+1} + \wt z_2 b_{\D_{r+1}}(2) y_{r+1} + \wt w_3 b_{\D_{r+1}}(2)/b_{r+3} & \wt z_2 y_{r+1} + \wt w_3 b_{\D_{r+1}}(2) & \wt w_3\end{matrix}\right],\]
with $x_{r+1}, y_{r+1}\in\C_0$, and $(\wt z_1, \wt z_2, \wt w_1, \wt w_2, \wt w_3)$ runing through $\C^5$.
In view of Lemma\,\ref{partial}, we have
\[P_{1,1}(M_{3,r}(\wh\L_r))= P_{3,3}(M_{3,r+1}(\wh\L_{r+1})).\]
Hence, for every $(z_2,w_2,w_3)\in\C^3$, there exists $(\wt z_1, \wt z_2, \wt w_2)\in\C^3$ such that
\[\left[\begin{matrix}
z_2x_r+w_2 b_{\D_r}(3) & w_2\\
z_2y_r +w_3 b_{\D_r}(2) & w_3\end{matrix}\right] =\hskip 6cm\]
\[\; 
\left[\begin{matrix}
\wt z_1 & \wt z_2\\
\wt z_1 x_{r+1} + \wt z_2 b_{\D_{r+1}}(3) x_{r+1}+ \wt w_2 b_{\D_{r+1}}(3)/b_{r+2} & \wt z_2 x_{r+1} + \wt w_2 b_{\D_{r+1}}(3)\end{matrix}\right].\]
The entrywise equations are
\begin{itemize}
\item[\textup{(1)}] $\; \wt z_1 = z_2 x_r + w_2 b_{\D_r}(3)$,
\item[\textup{(2)}] $\;\wt z_2=w_2$,
\item[\textup{(3)}] $\; \wt z_1 x_{r+1} + \wt z_2 b_{\D_{r+1}}(3) x_{r+1} + \wt w_2 b_{\D_{r+1}}(3)/b_{r+2}= z_2 y_r + w_3 b_{\D_r}(2)$,
\item[\textup{(4)}] $\;\wt z_2 x_{r+1} + \wt w_2 b_{\D_{r+1}}(3) = w_3.$
\end{itemize}
From (3) and (4) we may eliminate $\wt w_2$, obtaining
\begin{itemize}
\item[\textup{(5)}] $\; \wt z_1 b_{r+2}x_{r+1} + \wt z_2(b_{r+2}b_{\D_{r+1}}(3)-1)x_{r+1} = z_2 b_{r+2}y_r +w_3 (b_{r+2}b_{\D_r}(2)-1).$
\end{itemize}
Since $\D_r= \hbox{diag}(b_r,b_{r+1},b_{r+2})$ and $\D_{r+1}= \hbox{diag}(b_{r+1}, b_{r+2}, b_{r+3})$, it follows that
\[b_{r+2}b_{\D_{r+1}}(3)-1 = -\frac{b_{r+2}}{b_{r+1}} \quad \hbox{ and } \quad b_{r+2}b_{\D_r}(2)-1= - \frac{b_{r+2}}{b_r}.\]
Hence (5) takes the form\begin{itemize}
\item[\textup{($5'$)}] $\; \wt z_1 b_{r+2}x_{r+1} - \wt z_2 \frac{b_{r+2}}{b_{r+1}} x_{r+1} = z_2 b_{r+2} y_r - w_3 \frac{b_{r+2}}{b_r}.$
\end{itemize}
Simplifying by $b_{r+2}$ and substituting (1) and (2) into $(5')$ yields
\begin{itemize}
\item[\textup{(6)}] $\; (z_2 x_r + w_2 b_{\D_r}(3)) x_{r+1} - w_2 \frac{x_{r+1}}{b_{r+1}}= z_2 y_r - w_3 \frac{1}{b_r}$.
\end{itemize}
Multiplying (6) by $b_rb_{r+1}$, after rearrangement we obtain
\begin{itemize}
\item[\textup{($6'$)}] $\; z_2 b_rb_{r+1}(x_rx_{r+1}-y_r) - w_2 b_{r+1} x_{r+1} + w_3 b_{r+1}=0.$
\end{itemize}
Taking into account that $b_{r+1}x_{r+1}$ and $b_{r+1}$  are nonzero, we infer that
$(z_2, w_2,w_3)$ cannot be chosen arbitrarily in $\C^3$, what is a \emph{contradiction}.

\medskip
\noindent
\emph{Case 2.} Suppose now that $M_{3,r+1}(\wh\L_{r+1})$ is of type (T2).
The general element in $M_{3,r+1}(\wh\L_{r+1})$ is of the form
\[\left[\begin{matrix}
\wt z_1 & \wt z_2 & *\\
 t_{2,1} & t_{2,2}&*\\
*&*&*\end{matrix}\right],\]
where
\[t_{2,1}= \wt z_1 x_{r+1} + \wt z_2 b_{\D_{r+1}}(3) x_{r+1} + \wt w_1 q'_{r+1}+ \wt w_2\left(b_{\D_{r+1}}(3)/b_{r+2}-q'_{r+1}/x_{r+1}\right)\]
and
\[t_{2,2}=\wt z_2 x_{r+1}+ \wt w_1 q_{r+1} + \wt w_2\left(b_{\D_{r+1}}(3) -q_{r+1}/x_{r+1}\right),\]
with $x_{r+1}, q_{r+1}\in\C_0,\; q'_{r+1}\in\C$, and $(\wt z_1, \wt z_2, \wt w_1, \wt w_2, \wt w_3)$ running through $\C^5$.
In view of Lemma\,\ref{partial} we know that $\M_r= \wt\M_{r+1}\in\Lat M_2[\C]$, where
\[\M_r= \left\{ X(z_2,w_2,w_3): (z_2,w_2,w_3)\in\C^3\right\}\]
with
\[X(z_2,w_2,w_3)= \left[\begin{matrix}
z_2x_r + w_2 b_{\D_r}(3)& w_2\\
z_2 y_r + w_3 b_{\D_r}(2)& w_3\end{matrix}\right],\]
and
\[\wt\M_{r+1}= \left\{ \wt X(\wt z_1, \wt z_2, \wt w_1, \wt w_2): (\wt z_1, \wt z_2, \wt w_1, \wt w_2)\in\C^4\right\}\]
with
\[\wt X(\wt z_1, \wt z_2, \wt w_1, \wt w_2)=\left[\begin{matrix}
\wt z_1 & \wt z_2\\
t_{2,1}&t_{2,2} \end{matrix}\right].\]
The matrices
\[X(1,0,0)=\left[\begin{matrix}
x_r&0\\
y_r&0\end{matrix}\right], \; X(0,1,0)= \left[\begin{matrix}
b_{\D_r}(3)&1\\
0&0\end{matrix}\right], \; X(0,0,1)= \left[\begin{matrix}
0&0\\
b_{\D_r}(2)&1\end{matrix}\right]\]
are linearly independent in $M_2[\C]$, and so $\dim\M_r=3$.
On the other hand, we have
\[\wt X(1,0,0,0)= \left[\begin{matrix}
1&0\\
x_{r+1}& 0\end{matrix}\right], \quad \wt X(0,1,0,0)= \left[\begin{matrix}
0&1\\
b_{\D_{r+1}}(3) x_{r+1}& x_{r+1}\end{matrix}\right],\]
\[\wt X(0,0,1,0)= \left[\begin{matrix}
0&0\\
q'_{r+1}& q_{r+1}\end{matrix}\right],\]
\[\wt X(0,0,0,1)=
\left[\begin{matrix}
0&0\\
b_{\D_{r+1}}(3)/b_{r+2} - q'_{r+1}/x_{r+1} & b_{\D_{r+1}}(3)-q_{r+1}/x_{r+1}\end{matrix}\right].\]
It is clear that the subspace $\wt\M_{r+1}$ is of dimension 3 or 4.
Because of $\M_r= \wt\M_{r+1}, \; \dim\wt\M_{r+1}=3$ must hold, which happens if and only if 
\[(\wt X(0,0,1,0), \wt X(0,0,0,1))\]
 is linearly dependent.
Equivalently, we should have
\begin{eqnarray*}
0 &=& \det\left[\begin{matrix}
q'_{r+1}& q_{r+1}\\
b_{\D_{r+1}}(3)/b_{r+2} - q'_{r+1}/x_{r+1}& b_{\D_{r+1}}(3) - q_{r+1}/x_{r+1}\end{matrix}\right] \\
 &= & b_{\D_{r+1}}(3) \left(q'_{r+1}- \frac{q_{r+1}}{b_{r+2}}\right).
\end{eqnarray*}
Thus, $\dim\wt\M_{r+1}=3$ is valid exactly when
\[q'_{r+1}= \frac{q_{r+1}}{b_{r+2}}.\]
Let us assume now that $q'_{r+1}= q_{r+1}/b_{r+2}$ holds.
Then
\[\left(\wt X(1,0,0,0), \wt X(0,1,0,0), \wt X(0,0,1,0)\right)\]
forms a basis in $\wt\M_{r+1}$.
Hence the general element in $\wt\M_{r+1}$ is of the form
\[\wt X(\wt z_1, \wt z_2, \wt w_1,0)=
\left[\begin{matrix}
\wt z_1& \wt z_2\\
\wt z_1 x_{r+1} + \wt z_2 b_{\D_{r+1}}(3) x_{r+1}+ \wt w_1 q_{r+1}/ b_{r+2} & \wt z_2 x_{r+1} + \wt w_1 q_{r+1}\end{matrix}\right].\]
Given any $(z_2,w_2,w_3)\in\C^3$, there exists $(\wt z_1, \wt z_2, \wt w_1)\in\C^3$ such that
\begin{itemize}
\item[\textup{(1)}] $\; \wt z_1 = z_2 x_r+ w_2 b_{\D_r}(3)$,
\item[\textup{(2)}] $\;\wt z_2 = w_2$,
\item[\textup{(3)}] $\; \wt z_1 x_{r+1} + \wt z_2 b_{\D_{r+1}}(3) x_{r+1} + \wt w_1 \frac{q_{r+1}}{b_{r+2}}= z_2 y_r + w_3 b_{\D_r}(2)$,
\item[\textup{(4)}] $\; \wt z_2 x_{r+1} + \wt w_1 q_{r+1}= w_3$.
\end{itemize}
From (3) and (4) we may deduce
\begin{itemize}
\item[\textup{(5)}] $\; \wt z_1 b_r b_{r+1} x_{r+1} - \wt z_2 b_r x_{r+1} = z_2 b_r b_{r+1} y_r - w_3 b_{r+1}.$
\end{itemize}
Substituting (1) and (2) into (5) yields
\begin{itemize}
\item[\textup{(6)}] $\; z_2 b_rb_{r+1} (x_rx_{r+1}- y_r)- w_2 b_{r+1} x_{r+1} + w_3 b_{r+1} =0.$
\end{itemize}
Since $b_{r+1} x_{r+1}$ and $b_{r+1}$ are nonzero, the triplet $(z_2,w_2,w_3)$ cannot be arbitrarily chosen in $\C^3$, what is a \emph{contradiction}.

Summing up, the subspace $M_{3,r}(\wh\L_r)$ must be of type (T2), for every $r\in\N$.

\rightline{$\square$}

\section{Matching cross-sections of type (T2)}
\label{type2}

Investigating matching of neighboring $\D_r$-normal cross-sections of type (T2) we may deduce much simpler form of (T2) and we may derive connection between the corresponding parameters.

Given distinct numbers $b_1, b_2, b_3\in\C_0$, let $\D= \hbox{diag}(b_1,b_2,b_3)\in M_3[\C]$.
The subspace $\L\in\Lat M_3[\C]$ is called \emph{strongly $\D$-normal}, if $\dim\L=5$ and there exists a basis $(L_1,L_2,Q_1,Q_2,Q_3)$ in $\L$ of the form
\[L_1=\left[\begin{matrix}
1&0&0\\
x&0&0\\
y&0&0\end{matrix}\right], \quad L_2= \left[\begin{matrix}
0&1&0\\
b_\D(3)x&x&0\\
b_\D(2)y&y&0\end{matrix}\right],\]
\[Q_1= \left[\begin{matrix}
0&0&1\\
q/b_2& q&0\\
-y/(b_1b_3)& -y/b_1&0\end{matrix}\right],\]
\[ Q_2=\left[\begin{matrix}
0&0&0\\
1/b_2^2& 1/b_2&1\\
0&0&0\end{matrix}\right], \quad Q_3= \left[\begin{matrix}
0&0&0\\
0&0&0\\
1/b_3^2& 1/b_3&1\end{matrix}\right],\]
\[\hbox{with parameters }\; x,y,q\in\C_0.\]

Let $T, \wh T$ and $\wh\L$ be as in Section\,\ref{shift}.
Our main result in this setting is the following statement.

\begin{thm}
\label{main}
The linear manifold $\wh\L$ is strongly $D$-normal, which means that the subspace $M_{3,r}(\wh\L_r)$ is strongly $\D_r$-normal for every $r\in\N$, with parameters $x_r, y_r, q_r\in\C_0$.
Furthermore, we have
\[x_{r+1}= \frac{y_r}{x_r} \quad (r\in\N).\]
\end{thm}

\noindent
{\bf Proof.} For any $r\in\N$, let us consider the subspace $M_{3,r}(\wh\L_r)\in\Lat M_3[\C]$.
We know by Theorem\,\ref{exclusion} that $M_{3,r}(\wh\L_r)$ is $\D_r$-normal of type (T2), with parameters $x_r, y_r, q_r\in\C_0,\; q'_r\in\C$.
Recall that $\D_r= \hbox{diag}(b_r,b_{r+1}, b_{r+2})\in M_3[\C]$.
The general element in $M_{3,r}(\wh\L_r)$ is
\[\left[\begin{matrix}
z_1&z_2&w_1\\
z_1x_r+z_2b_{\D_r}(3)x_r+w_1q'_r+p_r& z_2x_r + w_1q_r+w_2(b_{\D_r}(3)-q_r/x_r)& w_2\\
z_1y_r+z_2b_{\D_r}(2)y_r + w_1 s_r+w_3 t_r& z_2y_r+w_1q_ry_r/x_r+u_r& w_3
\end{matrix}\right],\]
where
\[p_r=w_2\left(b_{\D_r}(3)/b_{r+1}-q'_r/x_r\right),\]
\[s_r=q'_ry_r/x_r + b_{\D_r}(1) q_r y_r/x_r, \quad t_r= b_{\D_r}(2)/b_{r+2} -b_{\D_r}(1)q_r/x_r - q'_r/x_r,\]
\[u_r=w_3(b_{\D_r}(2)-q_r/x_r),\]
and $(z_1,z_2,w_1,w_2,w_3)\in\C^5$.
Then
\[\M_r:= P_{1,1}(M_{3,r}(\wh\L_r))=\left\{ X(z_2,w_1,w_2,w_3): (z_2,w_1,w_2,w_3)\in\C^4\right\},\]
where
\[X(z_2,w_1,w_2,w_3)= \left[\begin{matrix}
z_2x_r +w_1q_r + w_2(b_{\D_r}(3)-q_r/x_r) & w_2\\
z_2y_r+w_1q_ry_r/x_r+w_3(b_{\D_r}(2)-q_r/x_r)& w_3\end{matrix}\right];\]
and
\[\wt\M_{r+1}:= P_{3,3}(M_{3,r+1}(\wh\L_{r+1}))= \left\{\wt X(\wt z_1, \wt z_2, \wt w_1, \wt w_2): (\wt z_1, \wt z_2, \wt w_1, \wt w_2)\in\C^4\right\},\]
where
\[\wt X(\wt z_1,\wt z_2, \wt w_1, \wt w_2)= \left[\begin{matrix}\wt z_1& \wt z_2\\
\wt\x_{2,1}& \wt\x_{2,2}
 \end{matrix}\right],\]
with
\[\wt\x_{2,1}= \wt z_1x_{r+1}+ \wt z_2 b_{\D_{r+1}}(3)x_{r+1} + \wt w_1 q'_{r+1} + \wt w_2(b_{\D_{r+1}}(3)/b_{r+2}-q'_{r+1}/x_{r+1}),\]
\[\wt \x_{2,2}=\wt z_2 x_{r+1}+ \wt w_1 q_{r+1} +\wt w_2 (b_{\D_{r+1}}(3) -q_{r+1}/x_{r+1}).\]
The system
\[X(1,0,0,0)= \left[\begin{matrix}
x_r&0\\
y_r&0\end{matrix}\right], \quad X(0,1,0,0)= \left[\begin{matrix}
q_r&0\\
q_ry_r/x_r&0\end{matrix}\right],\]
\[X(0,0,1,0)= \left[\begin{matrix}
b_{\D_r}(3)-q_r/x_r&1\\
0&0\end{matrix}\right], \; X(0,0,0,1)= \left[\begin{matrix}
0&0\\
b_{\D_r}(2)-q_r/x_r&1\end{matrix}\right]\]
is generating in $\M_r$.
Since $X(0,1,0,0)= \frac{q_r}{x_r} X(1,0,0,0)$, it follows that
\[\left(X(1,0,0,0),\, X(0,0,1,0), \, X(0,0,0,1)\right)\]
is a basis in $\M_r$.
So the general element in $\M_r$ is of the form
\[X(z_2,0,w_2,w_3)= \left[\begin{matrix}
z_2x_r + w_2(b_{\D_r}(3)-q_r/x_r) & w_2\\
z_2y_r + w_3(b_{\D_r}(2)-q_r/x_r)& w_3\end{matrix}\right],\]
where $(z_2,w_2,w_3)\in\C^3$ is arbitrary.

Similarly, the system
\[\wt X(1,0,0,0)= \left[\begin{matrix}1&0\\
x_{r+1}&0\end{matrix}\right], \quad \wt X(0,1,0,0)= \left[\begin{matrix}
0&1\\
b_{\D_{r+1}}(3)x_{r+1}& x_{r+1}\end{matrix}\right],\]
\[\wt X(0,0,1,0)=\left[\begin{matrix}
0&0\\
q'_{r+1}& q_{r+1}\end{matrix}\right],\]
\[ \wt X(0,0,0,1)= \left[\begin{matrix}
0&0\\
b_{\D_{r+1}}(3)/b_{r+2}-q'_{r+1}/x_{r+1}& b_{\D_{r+1}}(3)-q_{r+1}/x_{r+1}\end{matrix}\right]\]
is generating in $\wt\M_{r+1}$.
We know by Lemma\,\ref{partial} that $\M_r= \wt\M_{r+1}$.
Hence $\dim\wt\M_{r+1}=3$, which happens exactly when
\[(\wt X(0,0,1,0),\; \wt X(0,0,0,1))\]
is linearly dependent, that is when
\begin{eqnarray*}
0 & =& \det\left[\begin{matrix}
q'_{r+1}& q_{r+1}\\
b_{\D_{r+1}}(3)/b_{r+2}-q'_{r+1}/x_{r+1}& b_{\D_{r+1}}(3)-q_{r+1}/x_{r+1}\end{matrix}\right]\\
 &=& b_{\D_{r+1}}(3) \left(q'_{r+1} -\frac{q_{r+1}}{b_{r+2}}\right).
\end{eqnarray*}
We infer that
\[q'_{r+1} = \frac{q_{r+1}}{b_{r+2}}\]
must hold.

Since $(\wt X(1,0,0,0),\, \wt X(0,1,0,0),\, \wt X(0,0,1,0))$ is a basis, the general element in $\wt\M_{r+1}$ is of the form
\[\wt X(\wt z_1, \wt z_2, \wt w_1,0)= \left[\begin{matrix}
\wt z_1& \wt z_2\\
\wt z_1 x_{r+1} + \wt z_2 b_{\D_{r+1}}(3) x_{r+1} + \wt w_1 q_{r+1}/b_{r+2}& \wt z_2 x_{r+1}+ \wt w_1 q_{r+1}\end{matrix}\right],\]
where $(\wt z_1, \wt z_2, \wt w_1)\in\C^3$ is arbitrary.

Given any $(z_2,w_2,w_3)\in\C^3$, there exists $(\wt z_1, \wt z_2, \wt w_1)\in\C^3$ such that 
\[X(z_2,0,w_2,w_3)= \wt X(\wt z_1, \wt z_2, \wt w_1, 0).\]
The entrywise equations are
\begin{itemize}
\item[\textup{(1)}] $\;\wt z_1= z_2 x_r + w_2 (b_{\D_r}(3) - q_r/x_r),$
\item[\textup{(2)}] $\; \wt z_2=w_2$,
\item[\textup{(3)}] $\; \wt z_1 x_{r+1} + \wt z_2 b_{\D_{r+1}}(3) x_{r+1} + \wt w_1 q_{r+1}/b_{r+2}= z_2y_r +w_3(b_{\D_r}(2) -q_r/x_r)$,
\item[\textup{(4)}] $\; \wt z_2 x_{r+1} + \wt w_1 q_{r+1}=w_3$.
\end{itemize}
From (3) and (4) we may derive
\begin{itemize}
\item[\textup{(5)}] $\; \wt z_1b_rb_{r+1}x_{r+1} - \wt z_2 b_r x_{r+1} = z_2 b_rb_{r+1}y_r -w_3(b_{r+1}+b_rb_{r+1} q_r/x_r).$
\end{itemize}
Substituting (1) and (2) into (5) we may deduce
\begin{itemize}
\item[\textup{(6)}] $\; z_2b_rb_{r+1}(x_rx_{r+1}-y_r) - w_2 x_{r+1}b_{r+1} (1+ b_rq_r/x_r) + w_3 b_{r+1}(1+ b_r q_r/x_r)=0$.
\end{itemize}
However (6) can hold for every $(z_2,w_2,w_3)\in\C^3$ only if all coefficients are zero:
\[x_{r+1}= \frac{y_r}{x_r} \quad \hbox{ and } \quad \frac{q_r}{x_r}= - \frac{1}{b_r}.\]
Straightforward computations show that under the conditons
\[\frac{q_r}{x_r}= - \frac{1}{b_r} \quad \hbox{ and } \quad q'_r = \frac{q_r}{b_{r+1}},\]
the $\D_r$-normal subspace $M_{3,r}(\wh\L_r)$ is actually strongly $\D_r$-normal with parameters $x_r, y_r, q_r\in\C_0$.

\rightline{$\square$}

\bigskip

L\'aszl\'o K\'erchy

Bolyai Institute, University of Szeged

Aradi v\'ertan\'uk tere 1

Szeged, 6720 Hungary

e-mail: kerchy@math.u-szeged.hu

\end{document}